\documentclass[12pt]{amsart}
\usepackage{amssymb, amsmath, amscd, amsthm, setspace, latexsym, vruler, enumerate}
\usepackage[toc,page]{appendix}
\title[Representations of analytic functions as infinite products]{Representations of analytic functions as infinite
products and their application to numerical computations}
\author{Marcin Mazur}
\author{Bogdan V.~Petrenko}

\address{
Department of Mathematics \\
Binghamton University \\
P.O. Box 6000 \\
Binghamton, NY 13892-6000, USA } \email{
mazur@math.binghamton.edu}

\address{
Department of Mathematics \\ SUNY Brockport \\ 350 New Campus
Drive \\ Brockport, NY 14420, USA
 }

\email{ bpetrenk@brockport.edu}

\newtheorem{theorem}{Theorem}[section]
\newtheorem{lemma}[theorem]{Lemma}
\newtheorem{proposition}[theorem]{Proposition}
\newtheorem{corollary}[theorem]{Corollary}
\newtheorem{remark}[theorem]{Remark}

\newtheorem{example}[theorem]{Example}

\def\tr{\text{\rm tr}}

\def\Ln{\text{\rm Log}}

\def\tr{\text{tr}}

\newcommand{\mo}[1]{{ \hbox{\rm{ (mod\ $#1$) }}}}

\newtheorem*{te*}{Theorem}

\begin{document}
\maketitle

\begin{abstract}
Let $D$ be an open disk of radius $\le 1$ in $\mathbb C$, and let
$(\epsilon_n)$
be a sequence of $\pm 1$. We prove that for every analytic function $f: D \to
\mathbb C$ without
zeros in $D$, there exists
a unique sequence $(\alpha_n)$ of complex numbers such that
$f(z) = f(0)\prod_{n=1}^{\infty} (1+\epsilon_nz^n)^{\alpha_n}$ for every $z \in D$.
From this representation
we obtain a numerical method for calculating products of the form
$\prod_{p~\text{prime}} f(1/p)$
provided $f(0)=1$ and $f'(0) = 0$; our method generalizes a well known method of Pieter Moree. We
illustrate this method on a constant of Ramanujan $\pi^{-1/2}\prod_{p~\text{prime}}
\sqrt{p^2-p}\ln\left(p/(p-1)\right)$. From the properties of the exponents $\alpha_n$, we obtain a proof of
the following congruences, which have been the subject of several recent publications motivated by
some questions of Arnold:
for every $n \times n$ integral matrix $A$, every
prime number $p$, and every positive
integer $k$ we have $\text{tr} A^{p^k}  \equiv \text{tr}
A^{p^{k-1}}
\mo{p^k}$.

\vspace{4mm} \noindent {\bf Mathematics Subject Classification
(2010).} Primary 11Y60, 30E10, 30J99, 40A30, 40A20. Secondary 11A07, 11C20.

\vspace{3mm} \noindent {\bf Keywords.} Euler product, infinite products, approximating constants, congruences for traces.
\end{abstract}

\section{Introduction}
Many constants in number theory appear in the form $\prod_p f(1/p)$, where
the product is taken over all (sufficiently large) prime numbers and $f$ is a function analytic
in a neighborhood of $0$ and such that $f(0)=1$ and $f'(0)=0$. The results of this paper arose
from our attempt to compute some constants of this type with high accuracy. From this perspective,
our work should be considered as a generalization of the technique of Pieter Moree \cite{moree}, who
shows how to compute such products to high accuracy when $f$ is a rational function satisfying
certain additional properties.

To achieve our goal we prove the following result:

\vspace{3mm}
\noindent
{\bf Theorem.} {\em Let $f(z)$ be an analytic function without zeros in the disk $|z|<R\leq 1$ and let $(\epsilon_n)$
be a sequence of $\pm 1$. There exists unique sequence $(\alpha_n)$ such that the product
$f(0)\Pi_{n=1}^{\infty} (1+\epsilon_nz^n)^{\alpha_n}$ converges to $f$ uniformly on compact subsets of the disk $|z|<R$.
}

\vspace{3mm}
\noindent
(see Theorem~\ref{productrep}, where a simple way to compute the exponents $\alpha_n$ is stated).
As an illustration we mention the following formulas (see Examples \ref{example1},
\ref{example2}):
\[e^{-z}= \prod_{n=1}^{\infty}(1-z^n)^{\frac{\mu(n)}{n}}, \ \ \ \  e^{\frac{z}{z-1}}=\prod_{n=1}^{\infty}
(1-z^n)^{\frac{\phi(n)}{n}}, \]
valid for $|z|<1$, where $\mu$, $\phi$  are the M\"obius function and Euler function respectively.
After this work was completed, Pieter Moree kindly pointed to us an interesting paper
by G. Dahlquist \cite{euler}, where a product decomposition as in Theorem above
is obtained, and it is used to investigate analytic continuation of
certain Euler products. As a matter of fact, the idea to use such product decompositions for various functions
goes back to works of M\"obius \cite{mob} and later Landau (see, for example, \cite{landwal}).
Dahlquist's discussion of his result is somewhat brief, therefore we provide a detailed
proof including a careful analysis of the exponents $\alpha_n$. We hope that our paper will make this result more
widely known, which it fully deserves because of its
many applications. In addition to being instrumental for the numerical method described in Section~\ref{4},
it naturally leads to interesting arithmetic and combinatorial results by
studying the relations between the exponents $\alpha_n$ and the Taylor coefficients of $f$ and some related
functions. In Section 3, we will show an example
of such a result. Namely, we will obtain a short proof of the following theorem: if $A$ is an integral matrix,
$p$ is a prime number, and $n$ is a positive integer then the traces of $A^{p^n}$ and $A^{p^{n-1}}$ are congruent
modulo $p^n$. This result has been conjectured at the beginning of this century by V. I. Arnold, who
considered it as an analog of the classical Euler theorem. Arnold's conjecture  has been
the subject of several recent publications (among others, see \cite{a1}, \cite{a2}, \cite{vin}, \cite{zar2006},
\cite{zar2008}, \cite{mmbp}), where several different proofs can be found.
Let us mention yet another application. The products as in Theorem above have been considered as formal identities
in the theory of $q$-series, mainly in the case when the exponents $\alpha_n$ are integers (see \cite{andrews}).
Finally, the product decomposition has been used in a recent work \cite{primes}, where the authors have proved that
the regularized product of all prime numbers is equal to $4\pi^2$.

Returning to our original goal to get an accurate approximation of the product $\prod_p f(1/p)$, the strategy now is
the same as the one developed in \cite{moree}. Namely, using the decomposition
$f(z)=\Pi_{n=1}^{\infty} (1-z^n)^{\alpha_n}$ we get that
\[\prod_{p\geq t}f(1/p)=\prod_{p\geq t}\prod_{n=2}^\infty (1-p^{-n})^{\alpha_n}=\prod_{n=2}^\infty
\prod_{p\geq t}
(1-p^{-n})^{\alpha_n}=\prod_{n=2}^\infty\zeta(t,n)^{-\alpha_n}\]
where $t$ is sufficiently large and $\zeta(t,s)=\prod_{p\geq t}(1-p^{-s})^{-1}$ is a partial
zeta-function. The key observation behind our method is that the product on the right converges rapidly.
As an illustration, in Examples~\ref{e1} and \ref{e2} we compute first 50 decimal digits
of two constants appearing in analytic number theory. To the best of our knowledge,
this has not been done before. After writing a simple code, the computations
using PARI/GP and an ordinary laptop take only several seconds.

\paragraph*{\bf{Acknowledgments}}
We thank Maxim Korolev for providing us with information related to Example~\ref{e2}.
B.~Petrenko thanks Pieter Moree for very useful discussions of his work
\cite{moree} during Petrenko's visits to MPIM in
July-August of 2009 and in August 2010.

\section{Product decomposition}
For complex numbers $\alpha$ and $z$ such that $|z|<1$ we write $(1+z)^\alpha$ for
the binomial series $1+\sum_{n=1}^{\infty} {\alpha \choose n} z^n$. Then $(1+z)^{\alpha}=
e^{\alpha\Ln(1+z)}$, where $\Ln(1+z)=\sum_{n=1}^{\infty}(-1)^{n-1}z^n/n$ is the principal branch of
logarithm.

Recall that a sequence $(f_n)$ of functions on a topological space $X$ {\bf converges compactly}
to a function $f$ on $X$ if it converges uniformly to $f$ on every compact subset of $X$.
We will need the following well known consequence of the Residue Theorem:

\begin{lemma}\label{zero}
Let $(f_n)$ be a sequence of analytic functions in a domain $U$,
none of which assumes the value 0 in $U$. Suppose that this
sequence converges compactly on $U$ to a function $f$. Then either
$f=0$ or $f$ does not assume the value $0$ on $U$.
\end{lemma}
\begin{proof}
The function $f$ is analytic in $U$. Suppose that $f$ is not identically $0$.
Let $z_0\in U$. Then there is $\epsilon>0$ such that $f(z)\neq 0$ for all $z\neq z_0$ in
the disk $|z-z_0|\leq \epsilon$. It follows that $f_n'/f_n$ converges uniformly to $f'/f$
on the circle $\gamma$ with center $z_0$ and radius $ \epsilon$. Thus
$\int_{\gamma}f'(z)dz/f(z)=\lim_{n\to\infty} \int_{\gamma}f_n'(z)dz/f_n(z)=0$, as
$\int_{\gamma}f_n'(z)dz/f_n(z)=0$
(since $f_n'/f_n$ is analytic in $U$). It follows that $f(z_0)\neq 0$, as
otherwise the function $f'/f$ would have a simple pole at $z_0$ and the integral $\int_{\gamma}f'(z)dz/f(z)$
would not vanish.
\end{proof}

\begin{proposition}\label{product}
Consider the product $\Pi_{n=1}^{\infty} (1+\epsilon_nz^n)^{\alpha_n}$,
where $\epsilon_i=\pm 1$, $|z|<1$, and $\alpha_i\in \mathbb C$.
Let $0<R\leq 1$. The following conditions are equivalent.

\begin{enumerate}
\item The product converges compactly on the circle $|z|<R$.

\item The series $\sum_{n=1}^{\infty}
\alpha_n\Ln(1+\epsilon_nz^n)$ converges compactly on the circle
$|z|<R$.

\item The series $\sum_{n=1}^{\infty} \epsilon_n n \alpha_n z^{n-1}/(1+\epsilon_n z^n)$ converges
for $|z|< R$.

\item The series $\sum_{n=1}^{\infty} \alpha_n z^n$ converges for $|z|<R$.
\end{enumerate}
Moreover, if one of the above equivalent conditions holds, then the series in (2), (3), and (4) converge absolutely.
\end{proposition}
\begin{proof}
Assume (1). Then the product defines an analytic function $f(z)$ on $|z|<R$ which does not vanish
at any point by Lemma~\ref{zero}
. Thus $h(z)=\log f(z)$ exists and is analytic on $|z|<R$ (we take here the logarithm
satisfying $\log f(0)=0$). Let $h_n(z)=\sum_{k=1}^{n} \alpha_k\Ln(1+\epsilon_kz^k)$. Let $r<R$.
Then $|f(z)|>B$ for some $B>0$ and all $z$ such that $|z|\leq r$.
For any $\epsilon>0$ there is $n$ such that $|f(z)-e^{h_{N}(z)}|<B\epsilon$ for all $N\geq n$.
It follows that $|1-e^{h_N(z)-h(z)}|< \epsilon$ for all $z$ such that $|z|\leq r$. For
$\epsilon< 1/2$, this implies that $\Ln(1-\epsilon)<\Re(h_N(z)-h(z))<\Ln(1+\epsilon)$
and $\Im(h_N(z)-h(z))\in \bigcup_{k\in \mathbb Z}(-\arccos(1-\epsilon)+2k\pi,
\arccos(1-\epsilon)+2k\pi)$. Since $\Im(h_N(z)-h(z))$ is continuous on the connected set
$|z|<R$ and it vanishes at $0$, we must have $\Im(h_N(z)-h(z))\in (-\arccos(1-\epsilon),
\arccos(1-\epsilon))$. This implies that $h_N(z)$ converges uniformly to $h(z)$ on $|z|\leq r$.
Since $r$ is an arbitrary positive number less than $R$, we see that (2) holds.

That (2) implies (1) is an immediate consequence of the equality $e^{h_n(z)}=
\Pi_{k=1}^{n} (1+\epsilon_kz^k)^{\alpha_k}$.

The equivalence of (2) and (3) follows from the equality
$h_n'(z)=\sum_{k=1}^{n} \epsilon_kk\alpha_k z^{k-1}/(1+\epsilon_k
z^k)$ and the observation that if the series in (3) converges on
$|z|<R$ then it converges compactly.

Finally, the equivalence of (3) and (4) is a consequence of the inequalities
$|k\alpha_k z^{k-1}|/2\leq |\epsilon_k k \alpha_k z^{k-1}/(1+\epsilon_k z^k)|\leq
 |k\alpha_k z^{k-1}|/(1-|z|)$ and the remark that if the series in (3) or (4) converges
 for $|z|<R$ then it converges absolutely and compactly.

The absolute convergence of the series in (3) and (4) is clear. The absolute convergence of the
series in (2) follows from the absolute convergence of the series (4) and the inequality
$|\Ln(1+z)|\leq 2|z|$, which holds for all sufficiently small $z$ (for example, $|z|\leq 1/2$ works).
\end{proof}

\begin{lemma}\label{vanish}
Suppose that the product $f(z)=\Pi_{n=k}^{\infty}
(1+\epsilon_nz^n)^{\alpha_n}$ converges compactly on $|z|<R$ for
some $R$ such that $0<R\leq 1$. Then $f^{(i)}(0)=0$ for $1\leq i<k$.
\end{lemma}
\begin{proof}
Let $h(z)=\sum_{n=k}^{\infty} \alpha_n\Ln(1+\epsilon_nz^n)$. By
Proposition~\ref{product}, the series on the right converges
compactly to $h$, and $f=e^h$. It is clear that $h^{(i)}(0)=0$ for
$1\leq i<k$, as each summand has this property. The lemma follows now
by differentiation of the equality $f=e^h$.
\end{proof}

\begin{proposition}\label{formul}
Suppose that the product $f(z)=\Pi_{n=1}^{\infty} (1+\epsilon_nz^n)^{\alpha_n}$,
where $\epsilon_i=\pm 1$, converges compactly on $|z|<R$ for some $R>0$.
Then $f(z)=1+\sum_{n=1}^{\infty}b_nz^n$, where
\begin{equation}\label{formula}
b_n=\sum {\alpha_1 \choose k_1}{\alpha_2 \choose k_2}\ldots {\alpha_n \choose k_n}\epsilon_1^{k_1}
\epsilon_2^{k_2}\ldots\epsilon_n^{k_n}
\end{equation}
and the summation extends over all non-negative integers $k_1, k_2, \ldots, k_n$ such that
$k_1+2k_2+\ldots +nk_n=n$.
\end{proposition}
\begin{proof}
By Lemma~\ref{vanish}, the coefficient $b_n$ coincides with the coefficient at $z^n$ in the Taylor expansion
of $\Pi_{k=1}^{n} (1+\epsilon_kz^k)^{\alpha_k}$. The result follows now from the binomial
series expansion and the Cauchy formula for multiplying power series.
\end{proof}

\begin{corollary}\label{integral}
Suppose that the sequences $b_1,b_2,\ldots$ and $\alpha_1,\alpha_2,\ldots$ are related by
\text{\rm (\ref{formula})}. Then all the numbers $b_1,b_2,\ldots $ are integers if and only if all
the $\alpha_i$ are integers.
\end{corollary}
\begin{proof}
Since $\alpha \choose k$ are integers for any integers $\alpha$ and $k$, the integrality of the
$\alpha_i$'s implies the integrality of the $b_i$'s. Note that in the formula $(1)$ for $b_n$
the only contribution of $\alpha_n$ is the monomial $\epsilon_n \alpha_n$. This observation and
a straightforward induction on $n$ show that the integrality of $b_i$'s implies the integrality
of $\alpha_i$'s.
\end{proof}

\begin{lemma}\label{recursion}
Let $(H_n)$ be a sequence defined recursively by $H_1=1$ and $\displaystyle H_n=\sum_{d|n, d<n}H_d$.
Then $0<H_n\leq n^2$ for every $n$. In addition, if $(b_n)$ is any sequence and the sequence $(a_n)$
is defined by $a_n=\sum_{d|n}b_dH_{n/d}$, then $\displaystyle b_n=a_n-\sum_{d|n, d<n}a_d$.
\end{lemma}
\begin{proof}
Define a sequence $(J_n)$ by $J_1=1$ and $J_n=-1$ for $n>1$. The definition of $(H_n)$
is equivalent to the equality $(H_n)*(J_n)=(E_n)$, where $*$ denotes the Dirichlet
convolution, $E_1=1$, and $E_n=0$ for $n>1$.
This means that $(H_n)$ and $(J_n)$ are inverses of each other under the operation $*$.
Thus, if $(a_n)=(b_n)*(H_n)$, then $(b_n)=(a_n)*(J_n)$. This proves the second part of the lemma.
For the first part, recall that $\sum_{d|n}d^s\leq n^s\prod_{p|n}(1-p^{-s})^{-1}$ for $s>0$.
It follows that $\sum_{d|n}d^s\leq n^s\zeta(s)$, where $\zeta$ is the Riemann zeta function.
Let $s$ be such that $\zeta(s)\leq 2$. Then we claim that $H_n\leq n^s$. Indeed, this
is clear for $n=1$. Assuming that it holds for indexes less than $n$, we have
$H_n=\sum_{d|n, d<n}H_d\leq \sum_{d|n, d<n}d^s\leq \zeta(s)n^s-n^s\leq n^s$. Thus our claim follows
by induction. Taking $s=2$ yields the first part of the lemma, because $\zeta(2)=\pi^2/6<2$.
\end{proof}

\begin{lemma}\label{mainlemma}
Let $g(z)=\sum_{n=1}^{\infty}g_nz^{n-1}$ be analytic in the circle $|z|<R$ for some $0<R\leq 1$.
Let $\epsilon_i=\pm 1$.
Then there exists a unique
sequence $(\alpha_k)$ such that
\begin{equation}\label{exist}
 g(z)=\sum_{n=1}^{\infty} \epsilon_nn\alpha_n z^{n-1}/(1+\epsilon_n z^n)
\end{equation}
for all $z$ such that $|z|<R$. Moreover,
\begin{equation}\label{coef}
-g_n= \sum_{d|n}d\alpha_d (-\epsilon_d)^{n/d-1}.
\end{equation}
\end{lemma}
\begin{proof}
Note that
\[\frac{-\epsilon_nn\alpha_n z^{n}}{1+\epsilon_n z^n}=\sum_{k=1}^{\infty}(-\epsilon_n)^{k-1}n\alpha_nz^{nk}.
\]
Suppose first that (\ref{exist}) holds. Then the right hand side of $(\ref{exist})$ converges
compactly on $|z|<R$ and therefore
\[ -\sum_{n=1}^{\infty} g_nz^n=-zg(z)=\sum_{n=1}^{\infty}\sum_{k=1}^{\infty}
(-\epsilon_n)^{k-1}n\alpha_nz^{nk}=\]
\[=\sum_{n=1}^{\infty}\sum_{d|n}d\alpha_d (-\epsilon_d)^{n/d-1}z^n.\]
Comparing the coefficients at $z^n$ we get the formulas $(\ref{coef})$.
A straightforward induction establishes that for any sequence $(g_n)$ there is a unique
sequence $(\alpha_n)$ such that $(\ref{coef})$ holds for all $n$. This shows the uniqueness.
For the existence, it remains to show that the series $\sum_{n=1}^{\infty}
\epsilon_nn\alpha_n z^{n-1}/(1+\epsilon_n z^n)$ converges for
the sequence $(\alpha_n)$ defined by $(\ref{coef})$ and all $z$ such that $|z|<R$.
By Proposition~\ref{product}, it suffices to show that $\sum\alpha_n z^n$ converges
for $|z|<R$. We claim that $|n\alpha_n|\leq \sum_{d|n}|g_d|H_{n/d}$, where $(H_n)$ is the
sequence from Lemma~\ref{recursion}. Indeed, for $n=1$ this is clear. Assuming that it holds for
indexes less than $n$, we see that
\[ |n\alpha_n|\leq |g_n|+\sum_{d|n, d<n}d|\alpha_d|\leq |g_n|+\sum_{d|n, d<n}\sum_{e|d}|g_e|H_{d/e}=
\sum_{d|n}|g_d|H_{n/d}\]
(the last equality holds by Lemma~\ref{recursion}). Thus our claim follows by induction.
Since $H_n\leq n^2$ by Lemma~\ref{recursion}, we see that
\begin{equation}\label{pm}
|\alpha_n|\leq n\sum_{k=1}^n|g_k|.
\end{equation}
Since $R\leq 1$ and $\sum_n g_nz^n$ converges for $|z|<R$, each of the following series
also converges for $|z|<R$ by standard properties of power series: $\sum_n |g_n|z^n$, $\sum_n (\sum_{k=1}^n |g_k|)z^n$,
$\sum_n n(\sum_{k=1}^n |g_k|)z^n$. By (\ref{pm}), the convergence of the last series implies the convergence of
$\sum_n \alpha_n z^n$.
\end{proof}

\begin{remark}\label{remark} {\rm
The convergence of $\sum_n \alpha_n z^n$ can be obtained in a different way as follows.
Let $(\hat{\alpha}_n)$ be the sequence defined by $-g_n=\sum_{d|n}d\hat{\alpha}_d$ (which
is the sequence $(\alpha_n)$ obtained when $\epsilon_n=-1$ for all $n$). Then
the inequality $|n\hat{\alpha}_n|\leq \sum_{d|n}|g_d|$ follows easily from M\"obius inversion
formula. This implies the convergence of $\sum_n \hat{\alpha}_nz^n$. Now note the following
identity:
\[\frac{nz^{n-1}}{1-z^n}=\frac{nz^{n-1}}{1+\epsilon_nz^n}+
\frac{1+\epsilon_n}{2}\frac{2nz^{2n-1}}{1-z^{2n}}.
\]
Using this formula, we can rewrite the series
$\sum_{n=1}^{\infty} -n\hat{\alpha}_n z^{n-1}/(1- z^n)$ term by term, starting with $n=1$, into
$\sum_{n=1}^{\infty} \epsilon_nn\alpha_n z^{n-1}/(1+\epsilon_n z^n)$. It is not hard to see
that for $n=2^sm$, where $m$ is odd, the $\alpha_n$ obtained in this way is of the form
$\pm \hat{\alpha}_{m_1}\pm\ldots \pm\hat{\alpha}_{m_t}$, where $m_i=2^{s_i}m$ and
$s_1<s_2<\ldots<s_t\leq s$. This observation and the convergence of $\sum_n \hat{\alpha}_nz^n$
easily imply the convergence of $\sum_n \alpha_nz^n$.
}
\end{remark}

\begin{theorem}\label{productrep}
Let $f(z)$ be an analytic function without zeros in the disk $|z|<R\leq 1$ and let $(\epsilon_n)$
be a sequence of $\pm 1$. Then there exists a unique sequence $(\alpha_n)$ such that the product
$f(0)\Pi_{n=1}^{\infty} (1+\epsilon_nz^n)^{\alpha_n}$ converges compactly to $f$
on $|z|<R$. Moreover, if $f(z)=f(0)(1+\sum_{n=1}^{\infty}b_nz^n)$ and $f'(z)/f(z)=\sum_{n=1}^{\infty}g_nz^{n-1}$
then the following formulas hold:

\begin{equation}\label{formulaa}
b_n=\sum {\alpha_1 \choose k_1}{\alpha_2 \choose k_2}\ldots {\alpha_n \choose k_n}\epsilon_1^{k_1}
\epsilon_2^{k_2}\ldots\epsilon_n^{k_n},
\end{equation}

\begin{equation}\label{newton}
nb_n=g_n+\sum_{k=1}^{n-1}b_kg_{n-k},
\end{equation}

\begin{equation}\label{coeff}
-g_n= \sum_{d|n}d\alpha_d (-\epsilon_d)^{n/d-1}.
\end{equation}
\end{theorem}
\begin{proof}
We may assume that $f(0)=1$. The function $g(z)=f'(z)/f(z)$ is analytic in $|z|<R$. By Lemma~\ref{mainlemma},
there exists a unique sequence $(\alpha_k)$ such that
\begin{equation}
 g(z)=\sum_{n=1}^{\infty} \epsilon_nn\alpha_n z^{n-1}/(1+\epsilon_n z^n)
\end{equation}
for all $z$ such that $|z|<R$.
By Proposition~\ref{product}, we get $\log f(z)=\sum_{n=1}^{\infty} \alpha_n\Ln(1+\epsilon_nz^n)$
and $f(z)=\Pi_{n=1}^{\infty} (1+\epsilon_nz^n)^{\alpha_n}$.

Formula (\ref{formulaa}) has been obtained in Proposition~\ref{formul}. Formula (\ref{newton}) follows from
the equality $f'=fg$, i.e. from
\[\sum_{n=1}^{\infty}nb_nz^{n-1}=\left(1+\sum_{n=1}^{\infty}b_nz^n\right)\sum_{n=1}^{\infty}g_nz^{n-1}.
\]
Finally, (\ref{coeff}) has been established in (\ref{coef}) of Lemma~\ref{mainlemma}.
\end{proof}

\begin{remark}
{\rm The three natural choices for the sequence $(\epsilon_n)$ are $\epsilon_n=-1$ for all $n$,
$\epsilon_n=1$ for all $n$, and $\epsilon_n$ such that $\alpha_n$ has non-negative
real parts for all $n$. That the third choice always exists follows easily from the rewriting
procedure described in Remark~\ref{remark}. Unless some of the $\hat{\alpha}_n$'s are purely
imaginary, such a sequence $(\epsilon_n)$ is unique.
}
\end{remark}

\begin{remark}{\rm
Let $(\epsilon_n)$ be a sequence of $\pm 1$.
Starting with a function $f(z)=1+\sum_{n=1}^{\infty}b_n z^n$, analytic and without zeros
in $|z|<R$, we can compute the exponents $\alpha_n$ recursively in any one of the following ways:

\begin{enumerate}[\rm (i)]
\item Using formulas (\ref{formulaa}).

\item By Proposition~\ref{formul}, the sequence $(\alpha_n)$ is obtained recursively by the following
rule: $\epsilon_{n+1}\alpha_{n+1}$ is the coefficient at $z^{n+1}$ in the Taylor expansion of
$f(z)\Pi_{k=1}^{n} (1+\epsilon_kz^k)^{-\alpha_k}$.

\item Using (\ref{coeff}) and (\ref{newton}).
\end{enumerate}

It is intriguing that Corollary~\ref{integral} does not seem to be easily derivable just from (\ref{coeff}) and
(\ref{newton}), even though it is a straightforward consequence of (\ref{formulaa}).
}
\end{remark}

\begin{example}\label{example1}{\rm
We apply Theorem~\ref{productrep} to the exponential function $f(z)=e^z$ and $\epsilon_n=-1$ for all $n$.
Since $f'/f=1$, we
see that $g_1=1$ and $g_n=0$ for $n>1$. By (\ref{coeff}) and the M\"obius inversion formula we easily
get $\alpha_n=-\mu(n)/n$. Thus we have the following product expansion:
\begin{equation}\label{expo}
e^{-z}= \prod_{n=1}^{\infty}(1-z^n)^{\frac{\mu(n)}{n}}
\end{equation}
which converges for $|z|<1$. Formula (\ref{expo}) is not new. It has been stated already in \cite{mob}
(see formula (13) therein).
Now it is well known that $\sum_{n=1}^{\infty}\mu(n)/n=0$ (this equality, conjectured by Euler and proved
by von Mangoldt, is equivalent to the prime number theorem). Thus we may write
\begin{equation}\label{expo1}
e^{-z}= \prod_{n=1}^{\infty}\left(\frac{1-z^n}{1-z}\right)^{\frac{\mu(n)}{n}}.
\end{equation}
Taking $z=1$ leads to the equality
\[ e^{-1}=\prod_{n=1}^{\infty}n^{\frac{\mu(n)}{n}}\]
or, equivalently,
\begin{equation}\label{mu}
-1=\sum_{n=1}^{\infty}\frac{\mu(n)\ln n}{n}.
\end{equation}
Of course, what we did above is just a heuristic argument, as (\ref{expo1}) is valid only for $|z|<1$.
Nevertheless, (\ref{mu}) is correct and it has been stated by M\"obius \cite{mob} (who used heuristic
arguments similar to ours, see his formula (21)) and proved by E. Landau
\cite{lan1}. We hope that the above heuristic argument provides evidence that
the product decomposition established in Theorem~\ref{productrep} may be a source of interesting results in
number theory. Yet another application will be discussed in the next section.
}
\end{example}

\begin{example}\label{example2}{\rm
We apply Theorem~\ref{productrep} to the exponential function $f(z)=e^{\frac{z}{z-1}}$ and $\epsilon_n=-1$ for
all $n$. Since $f'/f=-1/(1-z)^2=\sum_{n=1}^{\infty}(-n)z^{n-1}$, we
see that $g_n=-n$ for all $n$. By (\ref{coeff}) and the M\"obius inversion formula we easily
get $\alpha_n=\phi(n)/n$. Thus we have the following product expansion:
\begin{equation}\label{expo2}
e^{\frac{z}{z-1}}= \prod_{n=1}^{\infty}(1-z^n)^{\frac{\phi(n)}{n}}
\end{equation}
which converges for $|z|<1$.
}
\end{example}

\section{Arnold's Conjecture}
The recursive formulas (\ref{newton}) are often called formulas of Newton. More precisely,
substituting  $q_n=-g_n$ we get
\begin{equation}\label{newton1}
 q_n+b_1q_{n-1}+\ldots +b_{n-1}q_1+nb_n=0.
\end{equation}
Newton observed in his {\em Arithmetica Universalis}, published in 1707, that when
$-b_1,\ldots, (-1)^kb_k$ are the elementary symmetric functions of $x_1, \ldots, x_k$
(and $b_n=0$ for $n>k$) and $q_n=x_1^n+\ldots + x_k^n$ then the relations (\ref{newton1})
hold.
Perhaps a bit less known are the following explicit formulas, which (in the case of symmetric
polynomials) go back to Girard (1629)
and Waring (1762):

\begin{equation}
q_n=n\sum (-1)^{k_1+k_2+\ldots +k_n}\frac{(k_1+k_2+\ldots +k_n-1)!}{k_1!k_2!\ldots k_n!}b_1^{k_1}
b_2^{k_2}\ldots b_n^{k_n},
\end{equation}

\begin{equation}
b_n=\sum\frac{(-1)^{k_1+k_2+\ldots +k_n}}{k_1!k_2!\ldots k_n!}\left(\frac{q_1}{1}\right)^{k_1}
\left(\frac{q_2}{2}\right)^{k_2}\ldots \left(\frac{q_n}{n}\right)^{k_n}.
\end{equation}
where, in both formulas, the summation extends over all non-negative integers $k_1,k_2,\ldots, k_n$
such that $k_1+2k_2+\ldots +nk_n=n$. See \cite{gould} for more about these formulas.
As observed by Moree in \cite{moree} (and by many others before), Newton's formulas relating the symmetric functions and
the power sums follow easily from (\ref{newton}). In fact, if $f(z)=1+b_1z+\ldots + b_kz^k$
is a polynomial, then $f(z)=(1-x_1z)(1-x_2z)\ldots (1-x_kz)$, where $x_1,\ldots, x_k$ are
the roots of the reciprocal polynomial $z^k+b_1z^{k-1}+\ldots + b_k$. It follows that
\[ \frac{f'(z)}{f(z)}=\sum_{j=1}^n \frac{-x_j}{1-x_jz}=\sum_{n=1}^{\infty}
 -(x_1^n+x_2^n+\ldots +x_k^n)z^{n-1}.\]
Thus $g_n=-(x_1^n+x_2^n+\ldots +x_k^n)$ and Newton's result follows from (\ref{newton}).

Now let us apply (\ref{coeff}) with $\epsilon_n=-1$ for
all $n$. By the M\"obius inversion formula, we get
\begin{equation}\label{integral1}
n\alpha_n=-\sum_{d|n}g_d\mu(n/d).
\end{equation}
Assume now that $b_1,\ldots, b_k$ are integers. Then, by Corollary~\ref{integral},
all $\alpha_n$ are integers too. Therefore we get the following result:

\begin{theorem}\label{arnold}
Let $x_1,\ldots, x_k$ be the roots of a monic integral polynomial $q(x)$ of degree $k$ and let
$q_n=x_1^n+x_2^n+\ldots +x_k^n$. Then
\[\sum_{d|n}q_d\mu(n/d)\equiv 0 \mo n.\]
\end{theorem}

Applying Theorem~\ref{arnold} when $q(x)$ is the characteristic  polynomial of an integral $k\times k$ matrix $A$
and $n=p^m$ is a power of a prime $p$, we get the following result.

\begin{theorem}\label{arnold1}
Let $A$ be an integral $k\times k$ matrix and $n=p^m$ be a power of a prime $p$. Then
$\tr A^{p^m}\equiv \tr A^{p^{m-1}}\mo{p^m}$.
\end{theorem}

Theorem~\ref{arnold1} has been conjectured by
Arnold (\cite{a1},\cite{a2}), and it has been the subject of several recent publications (\cite{vin},
\cite{zar2006}, \cite{zar2008}, \cite{mmbp}), even though it can be
found in papers going back to the 1920s (\cite{ja}, \cite{schur}). In \cite{mmbp} we
proved a more general result using a different method. However, the methods developed in the present paper
lead naturally to a discovery of Arnold's conjecture and the resulting proof is short and aesthetically
pleasing. We should mention that
already the paper of Moree \cite{moree} contains a similar proof of (\ref{integral1}), though
the above arithmetic consequences of this equality have not been addressed there.

\section{Numerical Method}\label{4}
In this section we describe the numerical method mentioned in the introduction.
We denote by $p_n$ the $n$th prime number.
Let $f$ be a function analytic and non-zero in the closed disk $|z|\leq R\leq 1$, $f(0)=1$, $f'(0)=0$.
Let $m$ be such that $Rp_m>1$.
Our goal is to approximate the product $\prod_{k=m}^{\infty}f(1/p_k)$ to high accuracy, as many constants
in number theory appear in such a form. Our strategy here is very similar to the one developed
by Moree \cite{moree} in the special case when $f$ is a rational function satisfying some additional properties.

By Theorem~\ref{productrep}, there is a  product decomposition
\[f(z)=\prod_{n=2}^\infty (1-z^n)^{\alpha_n}.\]
Thus
\[\prod_{k=m}^{\infty}f(1/p_k)=\prod_{k=m}^{\infty}\prod_{n=2}^\infty (1-p_k^{-n})^{\alpha_n}=\prod_{n=2}^\infty
\prod_{k=m}^{\infty}
(1-p_k^{-n})^{\alpha_n}=\prod_{n=2}^\infty\zeta_m(n)^{-\alpha_n}\]
where $\zeta_m(s)=\prod_{k=m}^{\infty}(1-p_k^{-s})^{-1}=\zeta(s)\prod_{k=1}^{m-1}(1-p_k^{-s})$
is a partial zeta-function (note that the
change in the order of multiplication is allowed as the product is absolutely convergent).
The key observation behind our method is that the product $\displaystyle \prod_{n=2}^\infty\zeta_m(n)^{-\alpha_n}$
converges rapidly.
More precisely, we have the following result.

\begin{theorem}\label{estimate}
Let $f(z)=\prod_{n=2}^\infty (1-z^n)^{\alpha_n}$ compactly converge in the disk $|z|<R\leq 1$.
Let $B$ be an upper bound for $|f'(z)/f(z)|$ on $|z|=R$.
Let $m$ be such that $Rp_m>1$ and let $M>m$ be such that
$\displaystyle C(R,B,m,M):= \frac{(e-1)Bp_m}{Rp_m-1}\frac{1}{(Rp_m)^M}\leq 1$. Then
\begin{equation}
\left|\prod_{k=m}^{\infty}f(1/p_k)-\prod_{n=2}^M\zeta_m(n)^{-\alpha_n}\right|\leq
C(R,B,m,M)\left|\prod_{n=2}^M\zeta_m(n)^{-\alpha_n}\right|.
\end{equation}
\end{theorem}
\begin{proof}
Let $f'/f=\sum_{n=1}^{\infty} g_n z^{n-1}$ and.
By Cauchy's inequality, we have $|g_n|\leq B/R^n$ for all $n\in \mathbb N$. By formula (\ref{coeff}) and
the M\"obius inversion formula, we have $n\alpha_n=-\sum_{d|n}g_d\mu(n/d)$. It follows that
\begin{equation}
|\alpha_n|\leq \frac{B}{R^n}
\end{equation}
In addition,
\[\zeta_m(n)-1\leq \sum_{k=p_m}^{\infty}\frac{1}{k^n}\leq p_m^{1-n}.
\]
for all $n\geq 3$.
Using these estimates and the inequality $\ln(1+x)\leq x$ (for $x>0$) we get
\[\left|\sum_{n=M+1}^{\infty}\alpha_n\ln\zeta_m(n)\right|\leq \sum_{n=M+1}^{\infty}\frac{B}{R^n}\ln\zeta_m(n)\leq\]
\[\leq B\sum_{n=M+1}^{\infty}\frac{p_m}{(Rp_m)^n}=\frac{Bp_m}{Rp_m-1}\frac{1}{(Rp_m)^M}.
\]
Note now that for $|z|\leq 1$ we have $|1-e^z|\leq (e-1)|z|$. It follows that
\[\left|1-\prod_{n=M+1}^\infty\zeta_m(n)^{-\alpha_n}\right|=\left|1-\exp\left(\sum_{n=M+1}^{\infty}-\alpha_n\ln\zeta_m(n) \right)
\right|\leq\]
\[\leq \frac{(e-1)Bp_m}{Rp_m-1}\frac{1}{(Rp_m)^M}= C(R,B,m,M)
\]
provided $M$ is such that $C(R,B,m,M)\leq 1$.
Since $|\prod_{k=m}^{\infty}f(1/p_k)|=|\prod_{n=2}^\infty\zeta_m(n)^{-\alpha_n}|$,
the theorem follows now by
multiplying the last inequality by
$|\prod_{n=2}^M\zeta_m(n)^{-\alpha_n}|$.
\end{proof}

The following two examples illustrate our method.

\begin{example}\label{e1}
{\rm  We will compute the first fifty decimal digits of the following constant
$A_1$ from the paper of Ramanujan \cite{raman}:

\[A_1=\lim_{n\to \infty}\frac{\sqrt{\ln n}}{n}\sum_{k=1}^{n}\frac{1}{d(k)}=\frac{1}{\sqrt{\pi}}\prod_{k=1}^{\infty}
\sqrt{p_k^2-p_k}\ln\left(\frac{p_k}{p_k-1}\right),\]
where $d(k)$ is the number of divisors of $k$.

Let $\displaystyle
f(z)=\frac{-\ln(1-z)}{z}\sqrt{1-z}$. Then $\sqrt{\pi}A_1=\prod_{k=1}^{\infty}f(1/p_k)$.
The function $f$ is holomorphic and non-zero in the unit disk. A straightforward computation yields
\[ \frac{f'(z)}{f(z)}= \frac{-1}{(1-z)\ln(1-z)}-\frac{1}{z}-\frac{1}{2(1-z)}.\]
Using the inequality $|(1-z)\ln(1-z)|\geq |z|-|z|^2$ we get
\[\left|\frac{f'(z)}{f(z)}\right|\leq \frac{1}{|z|-|z|^2}+\frac{1}{|z|}+\frac{1}{2(1-|z|)}=\frac{2}{|z|}+
\frac{3}{2(1-|z|)}.\]
This gives an estimate $|f'/f|\leq 18$ for $|z|=0.9$. Thus we may take $R=0.9$, $B=18$, and $m=7$,
so $p_m=17$ and $Rp_m>15$. Take $M=48$. Then
\[ C(R,B,m,M)
\leq \frac{38}{15^{48}}\leq 10^{-54}.\]
Thus, by Theorem~\ref{estimate}
\[ \left|A_1-\frac{1}{\sqrt{\pi}}\prod_{k=1}^{6}f(1/p_k)\prod_{n=2}^{48}\zeta_7(n)^{-\alpha_n}\right|\leq
10^{-54}\left|\frac{1}{\sqrt{\pi}}\prod_{k=1}^{6}f(1/p_k)\prod_{n=2}^{48}\zeta_7(n)^{-\alpha_n}\right|.
 \]
Now we calculate the exponents $\alpha_n$, $n=2,3,\ldots, 48$ and the product
$\prod_{k=1}^{6}f(1/p_k)\prod_{n=2}^{48}\zeta_7(n)^{-\alpha_n}$ accurate to 54 decimal digits
and get \\
$A_1=0.54685595528047446684551710099076178991021048592974\ldots$ (the computation has been done with 211
accurate digits using PARI/GP).

}
\end{example}

\begin{example}\label{e2}
{\rm Let $d(k)$ and $\sigma(k)$ denote the number and the sum of divisors of $k$, respectively.
Then $\sigma(k)/d(k)$ is the average divisor of $k$. The average of the average divisor
is then the quantity $\displaystyle \frac{1}{n}\sum_{k=1}^{n}\sigma(k)/d(k)$.  In \cite[Thm. 4.1]{erdos}
it is proved that this quantity is asymptotically equal to $cn/\ln n$, where
\[ c=\frac{1}{\sqrt{\pi}}\prod_{k=1}^{\infty}
\frac{p_k^{3/2}}{\sqrt{p_k-1}}\ln\left(1+\frac{1}{p_k}\right)=\frac{1}{\sqrt{\pi}}\prod_{k=1}^{\infty}f(1/p_k),
\]
where $\displaystyle f(z)=\frac{\ln(1+ z)}{z\sqrt{1-z}}$. V. I. Arnold, in his recent book \cite{a3},
attributes this asymptotic to A. A. Karatsuba
(note however that the formula for $c_1=c$ in the footnote on page 78 of \cite{a3} is incorrect: it has a factor of $1/\pi$ instead
of $1/\sqrt{\pi}$). According to Arnold, M. Korolev computed $c\approx 0.7138067\ldots$. We will see that
only the first 5 digits are accurate.
The function $f$ is holomorphic and non-zero in the unit disk and
\[ \frac{f'(z)}{f(z)}= \frac{1}{(1+z)\ln(1+z)}-\frac{1}{z}+\frac{1}{2(1-z)}.\]
The same estimates as in Example~\ref{e1} allow us to take $R=0.9$, $B=18$, $m=7$, and $M=48$ and get
\[ \left|c_1-\frac{1}{\sqrt{\pi}}\prod_{k=1}^{6}f(1/p_k)\prod_{n=2}^{48}\zeta_7(n)^{-\alpha_n}\right|\leq
10^{-54}\left|\frac{1}{\sqrt{\pi}}\prod_{k=1}^{6}f(1/p_k)\prod_{n=2}^{48}\zeta_7(n)^{-\alpha_n}\right|.\]
Now we calculate the exponents $\alpha_n$, $n=2,3,\ldots, 48$ and the product
$\prod_{k=1}^{6}f(1/p_k)\prod_{n=2}^{48}\zeta_7(n)^{-\alpha_n}$ accurate to 54 digits
and get \\
$c=0.71380993049991415224401060402799291827213336525147\ldots$.
}
\end{example}

\end{document}